\documentclass[12pt,a4paper]{amsart}
\usepackage[a4paper, left=28mm, right=28mm, top=28mm, bottom=34mm]{geometry}
\usepackage[all]{xy}
\usepackage{amsmath}
\usepackage{amssymb}
\usepackage{amsthm}
\usepackage{url}
\usepackage{mathrsfs}
\theoremstyle{definition}
\newtheorem{thm}{Theorem}[section]
\newtheorem{lem}[thm]{Lemma}

\newtheorem{prop}[thm]{Proposition}

\theoremstyle{definition}

\newtheorem{rem}[thm]{Remark}

\newtheorem{dfn}[thm]{Definition}

\numberwithin{equation}{section}
  
\begin{document}

\title[Elliptic curves with a specified subgroup and the trace formula]{Elliptic curves over a finite field with a specified subgroup and the trace formula}

\author{Tadahiro Katsuoka}
\address{Department of Mathematics, Faculty of Science, Kyoto University, Kyoto 606-8502, Japan}
\email{katsuoka.tadahiro.w30@kyoto-u.jp}

\date{February 25, 2025}



\begin{abstract}
Ihara and Birch obtained a formula expressing the sum of powers of the traces of elliptic curves over a fixed finite field of characteristic $p$ in terms of the traces of Hecke operators for $\mathrm{SL}_2(\mathbb{Z})$. Generalizing the theorems of Ihara and Birch, for a finite abelian group $A$ whose order is coprime to $p$, Kaplan and Petrow gave a formula for statistical description of powers of the traces of elliptic curves which contain subgroups isomorphic to $A$. In this paper, we generalize the theorems of Ihara, Birch, and Kaplan--Petrow to the case where the order of $A$ is divisible by $p$.
\end{abstract}

\maketitle

\section{Introduction}\label{sec1}
The number of isomorphism classes of elliptic curves over a fixed finite field  is described by the class numbers of imaginary quadratic orders, and the sum of the class numbers also appears in the Eichler--Selberg trace formula.
Based on this, 
Ihara and Birch  gave a formula for the sum of powers of the traces of elliptic curves over a finite field in terms of the traces of Hecke operators for $\mathrm{SL}_2(\mathbb{Z})$ (see \cite[Equation 47]{Ihara:1967}, \cite[Equation 4]{Birch:1968}).

Let $p$ be a prime, and $q$ be a power of $p$. Let $\mathbb{F}_q$ denote a finite field with $q$ elements. 
Let $A$ be a finite abelian group. 
When the order of $A$ is coprime to $p$, 
Kaplan and Petrow \cite[Theorem 3]{Kaplan and Petrow:2017} gave a formula for statistical description of powers of the traces of elliptic curves $E$ over $\mathbb{F}_q$ such that there exists an injective homomorphism $A \hookrightarrow E(\mathbb{F}_q)$, in terms of the traces of Hecke operators for some congruence subgroups of $\mathrm{SL}_2(\mathbb{Z})$. In this paper, we generalize the theorems of Ihara, Birch, and Kaplan--Petrow to the case where the order of $A$ is divisible by $p$.

To state the main result of this paper, we introduce some notation.
Let $\mathcal{C}$ denote the set of $\mathbb{F}_q$--isomorphism classes of elliptic curves over $\mathbb{F}_q$. For an elliptic curve $E$ over $\mathbb{F}_q$, the $\mathbb{F}_q$--isomorphism class of $E$ is written by the same symbol $E$.
For an elliptic curve $E$ over $\mathbb{F}_q$, we define the trace of $E$ as $t_E :=q+1-\# E(\mathbb{F}_q)$.
By Hasse's theorem, we have $|t_E| \leq 2 \sqrt{q}$. For a finite abelian group $A$, we define a function $\Phi_A$ on $\mathcal{C}$ by 
\begin{equation*}
    \Phi_A(E) := 
    \begin{cases}
        1 & \text{if there exists an injective homomorphism  }A \hookrightarrow E(\mathbb{F}_q)\\
        0 & \text{otherwise.}
    \end{cases}
\end{equation*}
For an integer $j \geq 0$, we define the Chebyshev polynomial of the second kind $U_j(t)$ by 
    \begin{equation*}
        U_0(t) := 1, \quad
        U_1(t) := 2t, \quad
        U_{j}(t) := 2t U_{j-1}(t) - U_{j-2}(t) \quad \text{for } j \geq 2.
    \end{equation*}
For an integer $k \geq 2$, we define the normalized Chebyshev polynomial $U_{k-2}(t,q) \in \mathbb{Z} [q,t]$ by
    \begin{equation*}
        U_{k-2}(t,q) := q^{k/2-1}U_{k-2}\left( \frac{t}{2\sqrt{q}} \right).
    \end{equation*}
    We define the moment of elliptic curves whose group of $\mathbb{F}_q$--rational points contain $A$ by
\begin{equation*}
    \mathbb{E}_q (U_{k-2}(t_E,q)\Phi_A):= \frac{1}{q} \sum_{\substack{E \in \mathcal{C} \\ A \hookrightarrow E(\mathbb{F}_q)}} \frac{U_{k-2}(t_E,q)}{\# \mathrm{Aut}_{\mathbb{F}_q}(E)}.
\end{equation*}
Here the sum of the right hand side runs over $\mathbb{F}_q$--isomorphism classes of elliptic curves $E$ over $\mathbb{F}_q$ such that there exists an injective homomorphism $A \hookrightarrow E(\mathbb{F}_q)$.

Ihara, Birch, and Kaplan--Petrow gave formulas for the moment $\mathbb{E}_q (U_{k-2}(t_E,q)\Phi_A)$ in terms of the traces of Hecke operators (see \cite[Equation $47$]{Ihara:1967}, \cite[Equation $4$]{Birch:1968}, \cite[Theorem 3]{Kaplan and Petrow:2017}).
The purpose of this paper is to give a formula for $\mathbb{E}_q (U_{k-2}(t_E,q)\Phi_A)$ in terms of the traces of Hecke operators for some congruence groups of $\mathrm{SL}_2(\mathbb{Z})$ when the order $\#A$ is divisible by $p$. 

Let $A$ be a finite abelian group. Assume that there exists an isomorphism
\[A \cong \mathbb{Z} / p^rn_1 \mathbb{Z} \times \mathbb{Z}/n_2\mathbb{Z}, \]
where $n_2$ divides $n_1$, and $n_1$ is coprime to $p$.
It is well--known that if there exists an injective homomorphism $A \hookrightarrow E(\mathbb{F}_q)$, then $A$ is isomorphic to the above group.
We define the congruence subgroup $\Gamma(p^rn_1,n_2)$ of $\mathrm{SL}_2(\mathbb{Z})$ by
\[ \Gamma(p^rn_1,n_2) := \left\{\begin{bmatrix}
            a & b \\
            c& d \\
        \end{bmatrix} \in \mathrm{SL}_2(\mathbb{Z}) : c \equiv 0 \ (\bmod{p^rn_1n_2}),\ a \equiv d \equiv 1 \ (\bmod{p^rn_1}) \right\}.\]
Let $S_k(\Gamma(p^rn_1,n_2))$ denote the $\mathbb{C}$--vector space of cusp forms of weight $k$ with respect to $\Gamma(p^rn_1,n_2)$.
Let $d \in (\mathbb{Z}/p^rn_1\mathbb{Z})^{\times}$. 
By the Eichler--Selberg trace formula, the trace of the Hecke operator $T_q \langle d \rangle$ on $S_k(\Gamma(p^rn_1,n_2))$ is expressed as the sum of four terms as follows (see Theorem \ref{trace_gamma_N,M} for details).
\begin{equation*}
\begin{split}
    T_{\mathrm{trace}}(p^rn_1,n_2) &= \mathrm{Tr}(T_q \langle d \rangle \mid S_k(\Gamma(p^rn_1,n_2)))\\
    &=  T_{\mathrm{id}}(p^rn_1,n_2) - T_{\mathrm{ell}}(p^rn_1,n_2) - T_{\mathrm{hyp}}(p^rn_1,n_2) + T_{\mathrm{dual}}(p^rn_1,n_2).
\end{split}
\end{equation*}
Here the term $T_{\mathrm{ell}}$ is derived from the elliptic conjugacy classes of $\Gamma(p^rn_1,n_2)$, and expressed as the sum of the class numbers of imaginary quadratic orders.

For a positive integer $n \geq 1$, we define Euler's totient function $\varphi(n)$, the function $\psi(n)$, and the function $\phi(n)$ as follows.  
\begin{equation*}
    \varphi(n) := n \prod_{p \mid n} \left( 1-\dfrac{1}{p} \right),\quad
        \psi(n) := n\prod_{p \mid n} \left(1 + \dfrac{1}{p} \right),\quad
        \phi(n) := n \prod_{p \mid n} (-\varphi(p)).
\end{equation*}

To focus on the term about the class numbers, for an integer $\lambda \geq 1$ dividing $\gcd(d^2q -1,n_1)$,
we define $T_{p^rn_1,\lambda}(q,d)$ by
    \begin{equation*}
        T_{p^rn_1,\lambda}(q,d) := \frac{\psi((p^rn_1)^2/\lambda^2)\varphi(p^rn_1/\lambda)}{\psi((p^rn_1)^2) \varphi(p^rn_1)}T_{\mathrm{ell}}(p^rn_1,\lambda).
    \end{equation*}
    We have the following explicit expression for $T_{p^rn_1,\lambda}(q,d)$.
\[T_{p^rn_1,\lambda}(q,d) = \sum_{\Lambda \mid (L/ \lambda)} \varphi(\Lambda^2) \varphi\left( \frac{p^rn_1}{\lambda  \Lambda} \right) \sum_{t^2 <4q}U_{k-2}(t,q) H_{p^rn_1,\lambda \Lambda}(t,q,d),\]
where $a \mid b$ means $a$ divides $b$. Here, $H_{p^rn_1,\lambda \Lambda}(t,q,d)$ is defined as a sum of class numbers of quadratic orders (see Definition \ref{H_n_1,n_2} for details).

The following theorem is the main theorem of this paper and generalizes the results of Ihara, Birch, and Kaplan--Petrow \cite[Theorem 3]{Kaplan and Petrow:2017}.

\begin{thm}\label{MainTheorem}
Let $p$ be a prime, and $q$ be a power of $p$. Let $A = \mathbb{Z}/p^rn_1\mathbb{Z} \times \mathbb{Z}/n_2\mathbb{Z}$. Assume that $n_2 \mid n_1$, $\gcd(n_1,q)=1$, $r \geq 1$, and $k \geq 2$.
\begin{enumerate}
    \item If $q \equiv 1 \pmod{n_2}$, we have
    \begin{equation*}
    \begin{split}
        \mathbb{E}_q (U_{k-2}(t_E,q)\Phi_A) = \frac{1}{q \varphi(p^rn_1 / n_2)} \sum_{\nu \mid \frac{\gcd (q-1,p^rn_1)}{n_2}} \phi (\nu) T_{p^rn_1,n_2 \nu}(q,1).
    \end{split}
    \end{equation*}
    \item If $q \not\equiv 1 \pmod{n_2}$, we have
    \begin{equation*}
        \mathbb{E}_q (U_{k-2}(t_E,q)\Phi_A) =  0.
    \end{equation*}
\end{enumerate}
\end{thm}

\begin{rem}
    In the case of $(2)$ in Theorem \ref{MainTheorem}, by using the Weil pairing, we can check that there dose not exist an elliptic curve $E$ over $\mathbb{F}_q$ such that there exists an injective homomorphism $A \hookrightarrow E(\mathbb{F}_q)$.
\end{rem}

The oraganization of this paper is as follows. In Section $2$, we introduce some notation. In Section $3$, we recall some facts about the class numbers of imaginary quadratic orders and define a sum of class numbers. In Section $4$, we revisit modular forms and the Eichler--Selberg trace formula. In Section $5$, we calculate the elliptic term in the Eichler--Selberg trace formula. Section $6$ organizes some results for elliptic curves over a finite field. Finally, in Section $7$, we prove Theorem \ref{MainTheorem} combining the results on modular forms and elliptic curves.

\section{Notation}
In this section, we introduce some notation.
Let $\mathbb{Z}$, $\mathbb{Q}$, and $\mathbb{C}$ denote the set of all integers, the set of all rational numbers, and the set of all complex numbers, respectively. 
Let $p$ denote a fixed prime, and $l$ denote a prime.
Let $q$ denote a power of $p$, and $\mathbb{F}_q$ denote a finite field with $q$ elements. 
For positive integers $a,b \geq 1$, $a \mid b$ means that $a$ divides $b$.
Let $(a,b):= \mathrm{gcd}(a,b)$ denote the greatest common divisor of $a$ and $b$. Let $v_p(a)$ denote the $p$--adic valuation of $a$ such that $v_p(p) = 1$. 

Let $n_1,n_2 \geq 1$ be positive integers such that $n_2 \mid n_1$.
For a positive integer $\nu \geq 1$, if every prime factor $l$ of $\nu$ satisfies $v_l(\nu) = v_l(n_1)$, we write $\nu \mid \mid n_1$ and say that $\nu$ is a \textit{full divisor} of $n_1$.

Following Kaplan--Petrow, 
for positive integers $\mu, \nu \geq 1$, we write $\mu \prec \nu$ if the following conditions are satisfied:
    \begin{enumerate}
        \item For every prime factor $l$ of $\nu$, $\mu$ is divisible by $l$ but $\mu$ is not divisible by any other primes.
        \item For every prime factor $l$ of $\nu$, we have $v_l (\mu) \leq v_l (n_1 / n_2) -1$.
    \end{enumerate}
    We also write $\mu \underset{\{n_1,n_2\}}{\prec} \nu$ if we emphasize $n_1,n_2$ used in the condition (2). 
    
Let $c \geq 1$ be a positive integer and $a,b \in \mathbb{Z}$ be integers. We define indicator functions $\delta(a,b),\delta_c(a,b)$ as follows, respectively.
\begin{equation*}
    \delta(a,b) := \begin{cases}
        1 & \text{if } a=b\\
        0 & \text{if } a \neq b
    \end{cases}, \qquad \delta_c(a,b) :=\begin{cases}
        1 & \text{if } a \equiv b \pmod{c} \\
        0 & \text{if } a \not\equiv b \pmod{c}.
    \end{cases}
\end{equation*}

Let $n \geq 1$ be a positive integer.
Let $\sigma(n):=\sum_{d \mid n}d$ denote the sum of divisors of $n$.
We define Euler's totient function $\varphi(n)$, the function $\psi(n)$, the function $\phi(n)$ as follows.  
\begin{equation*}
    \varphi(n) := n \prod_{l \mid n} \left( 1-\dfrac{1}{l} \right),\quad
        \psi(n) := n\prod_{l \mid n} \left(1 + \dfrac{1}{l} \right),\quad
        \phi(n) := n \prod_{l \mid n} (-\varphi(l)).
\end{equation*}
The function $\phi(n)$ is the inverse of the function $\varphi(n^2)$ in the Dirichlet convolution.
Then, the following holds.
\begin{equation*}
    \sum_{d \mid n} \varphi(d^2) \phi\left( \dfrac{n}{d} \right) = \delta(n,1).
\end{equation*}
Let $\omega(n)$ denote the number of distinct prime divisors of $n$. We define the Liouville function $\lambda(n)$ by $\lambda(n) := (-1)^{\omega(n)}$.

We use the following notation concerning the summation.
For positive integers $N,m,t,q \geq 1$, we put
    \begin{equation*}
    \begin{split}
        S(N,m,t,q) := &\{ c \in (\mathbb{Z}/N\mathbb{Z})^{\times} : \text{there exists an element } \tilde{c} \in (\mathbb{Z}/N(N,m)\mathbb{Z})^{\times}\\
        & \quad \text{such that } \tilde{c} \equiv c \ (\bmod{N}) \ \text{and } \tilde{c}^2-t\tilde{c}+q \equiv 0 \ (\bmod{N(N,m)})\}.
    \end{split}
    \end{equation*}
    For a Dirichlet character $\chi$ modulo $N$, the sum $\displaystyle\sum_{c \in S(N,m,t,q)} \chi(c)$
    is also written as $\displaystyle \sum_{c\in (\mathbb{Z}/N\mathbb{Z})^{\times}}^{*m} \chi(c)$ in \cite[p.1343, (4.3)]{Kaplan and Petrow:2017}.
    
\section{Class numbers of imaginary quadratic orders}
    Let $d < 0$ be a negative integer such that $d \equiv 0 \ \text{or }1 \ (\bmod{4})$. Let $h(d)$ denote the class number of the unique quadratic order of discriminant $d$.
    We define the weighted class number $h_w(d)$ as follows.
    \begin{equation*}
        h_w (d) := 
        \begin{cases}
            h(d)/3 & (\text{if } d= -3) \\
            h(d)/2 & (\text{if } d= -4) \\
            h(d) & (\text{if } d \neq -3,-4). \\
        \end{cases}
    \end{equation*}
    For a negative integer $d <0$ such that $d \equiv 2 \ \text{or } 3 \ (\bmod{4})$, we put $h_w(d) := 0$.
    For a negative integer $\Delta < 0$ with $\Delta \equiv 0 \ \text{or }1 \ (\bmod{4})$, let
    \begin{equation*}
        H(\Delta) := \sum_{d^2 \mid \Delta} h_w \left( \frac{\Delta}{d^2} \right)
    \end{equation*}
    be the \textit{Hurwitz-Kronecker class number}. 
    
To define a function $D(t;n)$ which are used for congruence conditions, we introduce a necessary lemma.

In the following, let $0 \leq m \leq n$ be integers and $l$ be a prime.
Let $\pi:\mathbb{Z} / l^n \mathbb{Z} \to \mathbb{Z} / l^m \mathbb{Z}$ be the canonical map sending $x \ (\bmod{l^n})$ to $x \ (\bmod{l^m})$. Let $c \in \mathbb{Z} / l^m \mathbb{Z}$. If an element $\tilde{c} \in \mathbb{Z} / l^n \mathbb{Z}$ satisfies $\pi(\tilde{c}) = c$, we call $\tilde{c}$ a lift of $c$ modulo $l^n$.   
Let $c_0 \ (\bmod{l^n})$ be a lift of $c$. Then we can express all lifts of $c$ modulo $l^n$ as $c_0 + j l^m \ (\bmod{l^n})$ where $0 \leq j < l^{n-m}$.

\begin{lem}\label{Dq+D^-1}
    Let $0 \leq B \leq C$, $D \in (\mathbb{Z} / l^C \mathbb{Z})^{\times} $, and $D^2 q \equiv 1 \ (\bmod{l^B})$.
    For any $C \leq i \leq B+C$, $Dq + D^{-1} \ (\bmod{l^i})$ is independent of the choice of a lift of $D$ modulo $l^i$ and well-defined.
\end{lem}
\begin{proof}
    See \cite[Lemma 5]{Kaplan and Petrow:2017}.
\end{proof}

\begin{dfn}
    Let $n_1, n_2 \geq 1$ be positive integers with $n_2 \mid n_1$, $d \in (\mathbb{Z} / n_1 \mathbb{Z})^{\times}$ with $d^2q \equiv 1 \ (\bmod{n_2})$. 
    Let $\widetilde{d} \in (\mathbb{Z}/n_1n_2\mathbb{Z})^{\times}$ be a lift of $d$ modulo $n_1n_2$.
    For any positive integer $n \mid n_1n_2$, we define the function $D(t;n)$ by
    \begin{equation*}
        D(t;n) := \delta_n (\widetilde{d}q+\widetilde{d}^{-1},t) = 
        \begin{cases}
            1 & (\text{if } \widetilde{d}q+\widetilde{d}^{-1} \equiv t \ (\bmod{n})) \\
            0 & (\text{otherwise}). 
        \end{cases}
    \end{equation*}
    Note the by the Chinese remainder theorem and Lemma \ref{Dq+D^-1}, 
    $\widetilde{d}q+\widetilde{d}^{-1} \ (\bmod{n})$ is independent of the choice of $\widetilde{d}$.
\end{dfn}

\begin{rem}
    Note that even if $p \mid n_1$, $\widetilde{d}q+\widetilde{d}^{-1} \ (\bmod{n_1n_2})$ is independent of the choice of $\widetilde{d}$, since $(n_2,p)=1$ and we do not consider a lift of $d$ modulo a power of $p$. For a fixed $t$, if $(n,m) = 1$, we have $D(t;n)D(t;m) = D(t;nm)$.
When $d=1$, $D(t;n) = 1$ is the same as $n \mid q+1-t$ and we can check whether the order $\#E(\mathbb{F}_q)$ with $t_E = t$ is divisible by $n$. 
\end{rem} 

\begin{dfn}\label{D_nu_mu(t)}
    Let $n_1,n_2 \geq 1$ be positive integers with $n_2 \mid n_1$. Let $d \in (\mathbb{Z}/ n_1 \mathbb{Z})^{\times}$ with $d^2q \equiv 1 \ (\bmod{n_2})$.
    For positive integers $\nu,\mu \geq 1$, we define the function $D_{\nu,\mu}(t)$ (which also depends on $q,d,n_1,n_2$) as follows.
    \begin{equation*}
        D_{\nu,\mu}(t) := \prod_{\substack{l \mid \nu \\l\mid \mu}} (D(t;l^{v_l (n_1 n_2 \mu) -1}) - D(t;l^{v_l (n_1 n_2 \mu)})) \cdot \prod_{\substack{l \mid \nu \\ l \nmid \mu}} D(t;l^{v_l(n_1 n_2)}).
    \end{equation*}
    We also write $D_{\nu,\mu}(n_1,n_2,t)$ if we emphasize $n_1,n_2$.
\end{dfn}
If $d=1$ and $D_{\nu,\mu}(t) = 1$, we have $v_l (q+1-t) = v_l (n_1 n_2 \mu) -1$ for each prime divisor $l$ of $\mu$.
\begin{rem}
    Definition \ref{D_nu_mu(t)} is slightly different from the definition of $D_{\nu,\mu}(t)$ in Kaplan--Petrow's paper \cite[p.1337, Equation (3.6)]{Kaplan and Petrow:2017}.
    In this paper, we adopt the above Definition \ref{D_nu_mu(t)}.
    In \cite{Kaplan and Petrow:2017}, an explicit calculation is only performed in the case where both $\nu$ and $\mu$ are powers of $l$, and in that case, the definition $D_{\nu,\mu}(t)$ in \cite{Kaplan and Petrow:2017} agrees with the above definition.
    However, if there exists a prime divisor $l$ of $\nu$ satisfying  $l \nmid \mu$, it is necessary to modify the definition of $D_{\nu,\mu}(t)$ as in Definition \ref{D_nu_mu(t)}.
\end{rem}

We define sum of class numbers $H_{n_1,n_2}(t,q,d)$ to represent the number of $\mathbb{F}_q$--isomorphism classes of elliptic curves over a finite field.
\begin{dfn}\label{H_n_1,n_2}
Let $n_1,n_2 \geq 1$ be positive integers with $n_2 \mid n_1$. Let $d \in (\mathbb{Z}/ n_1 \mathbb{Z})^{\times}$ with $d^2q \equiv 1 \ (\bmod{n_2})$. We define the sum of class numbers $H_{n_1,n_2}(t,q,d)$ as follows.
\begin{equation}
    \begin{split}
        H_{n_1,n_2}(t,q,d) &:= \frac{1}{2} H \left( \frac{t^2-4q}{n_2^2} \right) \delta_{n_2} (d^2q,1)D(t;n_1 n_2) \\
        &\quad + \sum_{\substack{m \mid\mid n_1 \\ m \geq 2}} \sum_{\mu \prec m} \lambda(m) \frac{1}{2} H \left( \frac{t^2-4q}{(n_2 \mu)^2} \right) \delta_{n_2 \mu}(d^2q,1) D_{n_1,\mu}(t).
    \end{split}
\end{equation}
Here $\mu \prec m$ means $\mu \underset{\{n_1,n_2\}}{\prec} m$, and $D_{n_1,\mu}(t)$ means $D_{n_1,\mu}(n_1,n_2,t)$.
\end{dfn}

\begin{lem}\label{H_p^rn_1,n_2(t,q,1)}
    Assume that $(q,n_1) =1$ and $r \geq 1$.
    We have the following equation.
    \begin{equation*}
        H_{p^rn_1,n_2}(t,q,d) = D(t;p^r) H_{n_1,n_2}(t,q,d).
    \end{equation*}
\end{lem}
\begin{proof}
    By definition, we have
    \begin{equation*}
    \begin{split}
        &H_{p^rn_1,n_2}(t,q,d) \\
        &= \frac{1}{2} H \left( \frac{t^2-4q}{n_2^2} \right) \delta_{n_2}(d^2q,1) D(t;p^rn_1n_2) \\
        &\quad + \sum_{\substack{m \mid\mid p^rn_1 \\ m \geq 2}} \sum_{\mu \underset{\{p^rn_1,n_2\}}{\prec} m} \lambda(m) \frac{1}{2} H \left( \frac{t^2-4q}{(n_2 \mu^2)} \right) \delta_{n_2 \mu}(d^2q,1) D_{p^rn_1,\mu}(p^rn_1,n_2,t).
    \end{split}
    \end{equation*}
    Since $n_1$ and $n_2$ are coprime to $p$, we have
    \begin{equation*}
        D(t;p^rn_1n_2) = D(t;n_1n_2)D(t;p^r).
    \end{equation*}
    Furthermore, focusing on the condition of the second sum $\mu \underset{\{p^rn_1,n_2\}}{\prec} m$, when $p \mid \mu$, we have $\delta_{n_2\mu}(d^2q,1) = 0$.
    So it is sufficient to consider only those $\mu$ such that $p \nmid \mu$.
    Therefore, the first condition on the first sum can be written as $m \mid\mid n_1$ instead of $m \mid\mid p^r n_1$.
    Also, concerning the function $D_{p^r n_1, \mu}(p^r n_1, n_2, t)$, since $p$ does not divide $\mu$, we have
    \begin{equation*}
        D_{p^rn_1,\mu}(p^rn_1,n_2,t) = D_{n_1,\mu}(n_1,n_2,t) D(t;p^r).
    \end{equation*}
    Hence, the following holds.
    \begin{equation*}
    \begin{split}
        &H_{p^rn_1,n_2}(t,q,d) \\
        &= \frac{1}{2} H \left( \frac{t^2-4q}{n_2^2} \right) \delta_{n_2}(d^2q,1) D(t;n_1n_2)D(t;p^r) \\
        &\quad + \sum_{\substack{m \mid\mid n_1 \\ m \geq 2}} \sum_{\mu \underset{\{n_1,n_2\}}{\prec} m} \lambda(m) \frac{1}{2} H \left( \frac{t^2-4q}{(n_2 \mu^2)} \right) \delta_{n_2 \mu}(d^2q,1) D_{n_1,\mu}(n_1,n_2,t)D(t;p^r).
    \end{split}
    \end{equation*}
    Therefore, we have $H_{p^rn_1,n_2}(t,q,d) = D(t;p^r) H_{n_1,n_2}(t,q,d)$.
\end{proof}

\section{The Eichler--Selberg trace formula for $\Gamma_1(N)$ and $\Gamma(N,M)$}

The Eichler--Selberg trace formula for $\mathrm{SL}_2(\mathbb{Z})$ was shown by Selberg \cite{Selberg:1956}. When the index of the Hecke operator and N are coprime, the trace formula for $\Gamma_0(N)$ and $\chi$ was proved  by Hijikata \cite{Hijikata:1974}, and the general case was proved by Oesterl\'{e} \cite{Oesterle:1977}.
In this section, we first define linear operators on $S_k(\Gamma(N, M))$. Then, we introduce the Eichler--Selberg trace formula, which is an important formula for the trace of the Hecke operator $T_q$ on $S_k(\Gamma_1(N), \chi)$, and introduce the Eichler--Selberg trace formula for $\Gamma(N, M)$. We use the same notation for modular forms as in \cite[Sections 5.1 and 5.2]{Diamond and Shurman:2005}.

In this section, let $M,N \geq 1$ be positive integers with $M \mid N$.
We define $\mathrm{SL}_2(\mathbb{Z})$ and $\mathrm{GL}_2^+ (\mathbb{Q})$ as follows.
\begin{equation*}
    \mathrm{SL}_2(\mathbb{Z}) := \left\{ \begin{bmatrix}
        a & b \\
        c & d\\
    \end{bmatrix} : a,b,c,d \in \mathbb{Z} \quad \text{and } ad-bc = 1 \right\}.
\end{equation*}
\begin{equation*}
    \mathrm{GL}_2^+ (\mathbb{Q}) := \left\{ \begin{bmatrix}
        a & b \\
        c& d \\
    \end{bmatrix} : a,b,c,d \in \mathbb{Q} \quad \text{and } ad-bc>0 \right\}.
\end{equation*}

We define congruence subgroups $\Gamma_0(N)$, $\Gamma_1(N)$, $\Gamma(N,M)$ of $\mathrm{SL}_2(\mathbb{Z})$ as follows, respectively. 
\begin{equation*}
    \begin{split}
        \Gamma_0(N) &:= \left\{\begin{bmatrix}
            a & b \\
            c& d \\
        \end{bmatrix} \in \mathrm{SL}_2(\mathbb{Z}) : c \equiv 0 \ (\bmod{N}) \right\},\\
        \Gamma_1(N) &:= \left\{\begin{bmatrix}
            a & b \\
            c& d \\
        \end{bmatrix} \in \mathrm{SL}_2(\mathbb{Z}) : c \equiv 0 \ (\bmod{N}) \ \text{and } a \equiv d \equiv 1 \ (\bmod{N}) \right\},\\
        \Gamma(N,M) &:= \left\{\begin{bmatrix}
            a & b \\
            c& d \\
        \end{bmatrix} \in \mathrm{SL}_2(\mathbb{Z}) : c \equiv 0 \ (\bmod{MN}) \ \text{and } a \equiv d \equiv 1 \ (\bmod{N}) \right\}.
    \end{split}
\end{equation*}

From the exact sequence
\begin{equation*}
    1 \to \Gamma(N,M) \to \Gamma_0 (MN) \to (\mathbb{Z}/N \mathbb{Z})^{\times} \to 1,
\end{equation*}
the action of $\Gamma_0(MN)$ on $S_k(\Gamma(N, M))$ can be identified with the action of $(\mathbb{Z}/N\mathbb{Z})^{\times}$.
Therefore, we can define the diamond operator $\langle d \rangle$ in the same way as in the case of $\Gamma_1(N)$.
For each $d \in (\mathbb{Z} / N\mathbb{Z})^{\times}$, take an element $\gamma = \begin{bmatrix}
        a & b \\
        c & \delta \\
    \end{bmatrix} \in \Gamma_0(N)$ with $\delta \equiv d \ (\bmod{N})$.
We define the diamond operator
\begin{equation*}
    \langle d \rangle: S_k(\Gamma(N, M)) \to S_k(\Gamma(N, M))
\end{equation*}
by $\langle d \rangle f := f[\gamma]_k$. 
It is easy to see that $\langle d \rangle$ is independent of the choice of $\gamma$.

Moreover, for each prime $p$, we consider the following double coset decomposition:
\begin{equation*}
    \Gamma(N, M) \begin{bmatrix}
        1 & 0 \\
        0 & p \\
    \end{bmatrix} \Gamma(N, M) = \bigsqcup_j \Gamma(N, M) \beta_j, \quad \text{where } \beta_j \in \mathrm{GL}_2^+ (\mathbb{Q}).
\end{equation*}
Using this decomposition, we define the Hecke operator
\begin{equation*}
    T_p: S_k(\Gamma(N, M)) \to S_k(\Gamma(N, M))
\end{equation*}
by $T_p(f) := \sum_j f[\beta_j]_k$.

We can check that $\langle d \rangle$ and $T_p$ commute with each other.
Let $v \geq 2$ be a positive integer. For a positive integer $v \geq 2$, we define the Hecke operator $T_{p^v}$ inductively as follows:
\begin{equation*}
    T_{p^v} := T_p T_{p^{v-1}} - p^{k-1} \langle p \rangle T_{p^{v-2}}
\end{equation*}
 where $T_1$ is the identity map.

Let $\chi: (\mathbb{Z}/N\mathbb{Z})^{\times} \to \mathbb{C}^{\times}$ be a Dirichlet character modulo $N$. 
We define the $\chi$-eigenspace $S_k(\Gamma(N, M), \chi)$ of $S_k(\Gamma(N, M))$ as follows.
\begin{equation*}
    S_k(\Gamma(N,M),\chi) := \{f \in S_k(\Gamma(N,M)) : \langle d \rangle f = \chi(d)f \quad \text{for all } d \in (\mathbb{Z} / N\mathbb{Z})^{\times}\}.
\end{equation*}
Similarly, we define the $\chi$-eigenspace of $S_k(\Gamma_1(N), \chi)$ of $S_k(\Gamma_1(N))$ as follows.
\begin{equation*}
    S_k(\Gamma_1(N),\chi) := \{f \in S_k(\Gamma_1(N)) : \langle d \rangle f = \chi(d)f \quad \text{for all } d \in (\mathbb{Z} / N\mathbb{Z})^{\times}\}.
\end{equation*}
The spaces of modular forms $S_k(\Gamma_1(N))$, $S_k(\Gamma(N, M))$ have the following eigenspace decomposition:
\begin{equation*}
     S_k(\Gamma_1(N)) = \bigoplus_{\chi \bmod{N}} S_k(\Gamma_1(N),\chi), \qquad S_k(\Gamma(N,M)) = \bigoplus_{\chi \bmod{N}} S_k(\Gamma(N,M),\chi).
\end{equation*}
Here ``$\chi \bmod{N}$" in the right hand side means that $\chi$ runs through all the Dirichlet characters modulo $N$. 

The following properties can be easily checked.
\begin{enumerate}
    \item Let $\chi, \widetilde{\chi}$ be Dirichlet characters modulo $N,MN$ respectively. Assume that $\widetilde{\chi} \ (\bmod{N})$ coincides with $\chi$. In this case, the following holds:
    \begin{equation*}
        S_k(\Gamma(N,M), \chi) = S_k(\Gamma_1(MN),\widetilde{\chi}).
    \end{equation*}
    \item Let $\chi, \widetilde{\chi}$ be Dirichlet characters modulo $N,MN$ respectively. Assume that $\widetilde{\chi} \ (\bmod{N})$ coincides with $\chi$. Let $T_p$ be the Hecke operator on $S_k(\Gamma(N, M), \chi)$, and $T_p'$ be the Hecke operator on $S_k(\Gamma_1(MN), \widetilde{\chi})$. Then, we have $T_{p^r} = T_{p^r}'$ for every $r \geq 1$.
    \item We have the following equation for every $r \geq 1$.
    \begin{equation}\label{sum_ch}
    \mathrm{Tr} (\langle d \rangle T_{p^r} \mid S_k(\Gamma(N,M))) =
    \sum_{\chi \bmod{N}} \chi(d) \mathrm{Tr}(T_{p^r} \mid S_k(\Gamma(N,M),\chi)).\\
\end{equation}
\end{enumerate}

The following formula is called the Eichler--Selberg trace formula.
\begin{thm}[\textbf{Eichler--Selberg trace formula for $\Gamma_1(N)$} {\cite[p.370,  Statement of the final result]{Knightly and Li}}]\label{trace_gamma_0}
Let $(q,N)=1$, $k \geq 2$. Let $\chi$ be a Dirichlet character modulo $N$ with $\chi(-1) = (-1)^k$. We have the following formula.
\begin{equation}
\begin{split}
\mathrm{Tr} (T_q \mid S_k(\Gamma_1(N),\chi)) = &\frac{k-1}{12} \psi(N) \chi(q^{1/2})q^{k/2-1} \\
& -\frac{1}{2} \sum_{t^2 < 4q} U_{k-2}(t,q) \sum_{m^2 \mid (t^2-4q)} h_w \left( \frac{t^2-4q}{m^2} \right) \mu_\chi(t,m,q) \\
& -\frac{1}{2} \sum_{b\mid q} \min(b,q/b)^{k-1} \sum_\tau(\varphi((\tau,N/\tau))\chi(y_\tau)) \\
& + \delta(k,2) 1_{\chi = 1} \sum_{\substack{c \mid q \\ (N,q/c)=1}} c
\end{split}
\end{equation}
Here the terms in the right hand side are defined as follows.
\begin{itemize}
    \item If $q$ is not a perfect square, let $\chi(q^{1/2})=0$.
    \item $U_{k-2}(t,q)$ be the normalized Chebyshev polynomial (see Section $1$).
    \item $\mu_\chi (t,m,q)$ is defined by
    \begin{equation*}
        \mu_\chi (t,m,q) := \frac{\psi(N)}{\psi(N/(N,m))} \sum_{c \in S(N,m,t,q)} \chi(c).
    \end{equation*}
    Here the sum in the right hand side means 
    $c$ runs through all elements of $(\mathbb{Z}/N\mathbb{Z})^\times$ such that there exists a lift $\tilde{c} \ (\bmod{N(N,m)})$ of $c$ satisfying
    \begin{equation*}
        \tilde{c}^2-t\tilde{c}+q \equiv 0 \ (\bmod{N(N,m)})
    \end{equation*}
    (for the set $S(N,m,t,q))$, see Section $2$). 
    \item $\tau$ runs through all positive factors of $N$ such that $(\tau,N/\tau)$ divides $N/\mathrm{cond}(\chi)$ and $b-q/b$, where $\mathrm{cond}(\chi)$ denotes the conductor of $\chi$.
    \item $y_\tau$ is the unique element of $\mathbb{Z}/(N/(\tau,N/\tau))\mathbb{Z}$ such that $y_\tau \equiv d \ (\bmod{\tau})$ and $y_\tau \equiv q/d \ (\bmod(N/\tau))$.
    \item $1_{\chi=1}=1$ if $\chi$ is the trivial character, and $1_{\chi=1} = 0$ otherwise.
\end{itemize}
\end{thm}

 Kaplan--Petrow calculated the right hand side of (\ref{sum_ch}) by using Theorem \ref{trace_gamma_0} and proved the following theorem.
\begin{thm}[\textbf{Eichler--Selberg trace formula for $\Gamma(N,M)$} {\cite[Theorem 9]{Kaplan and Petrow:2017}}]\label{trace_gamma_N,M}
    Let $p$ be a prime, $q$ be a power of $p$, $M,N \geq 1$ be positive integers, $d \in (\mathbb{Z}/N\mathbb{Z})^{\times}$, and $k \geq 2$ be an integer. Assume that $M \mid N$, $(N,q)=1$, and $d^2 q \equiv 1 \ (\bmod{M})$. Let $T_q$ be the Hecke operator on $S_k(\Gamma(N,M))$ and $\langle d \rangle$ be the diamond operator on $S_k(\Gamma(N,M))$.
    Let $L=(d^2 q -1,N)$. Then, we have the following formula.
    \begin{equation*}
\begin{split}
    T_{\mathrm{trace}}(N,M) &= \mathrm{Tr}(T_q \langle d \rangle \mid S_k(\Gamma(N,M)))\\
    &=  T_{\mathrm{id}}(N,M) - T_{\mathrm{ell}}(N,M) - T_{\mathrm{hyp}}(N,M) + T_{\mathrm{dual}}(N,M),
\end{split}
\end{equation*}
where $T_{\mathrm{id}}, T_{\mathrm{ell}}, T_{\mathrm{hyp}}, T_{\mathrm{dual}}$ are respectively defined by  
\begin{align*}
    T_{\mathrm{id}}(N,M)&:= \varphi(N) \cdot\frac{k-1}{24}q^{k/2-1} \psi(NM) (\delta_N (q^{1/2}d,1) + (-1)^k \delta_N(q^{1/2}d,-1)),\\
    T_{\mathrm{ell}}(N,M) &:= \varphi(N) \cdot \frac{\psi(N^2)}{\psi(N^2/M^2)} \\
    & \quad \cdot \sum_{\Lambda \mid (L/M)} \frac{\varphi(\Lambda^2)\varphi(N/(M\Lambda))}{\varphi(N/M)} \sum_{t^2 < 4q} U_{k-2}(t,q) H_{N,\Lambda M}(t,q,d), \\
    T_{\mathrm{hyp}}(N,M) &:= \varphi(N) \cdot \frac{1}{4} \sum_{b \mid q} \min(b,q/b)^{k-1} \sum_{\substack{\tau \mid NM \\ g \mid (b-q/b)}}\bigg( \frac{\varphi(g) \varphi(N(M,g)/g)}{\varphi(N)}\\
            &\quad \times (\delta_{N(M,g)/g} (y_{\tau}d,1) + (-1)^k \delta_{N(M,g)/g}(y_\tau d ,-1)) \bigg),\\
    T_{\mathrm{dual}}(N,M) &:=\sigma(q)\delta(k,2).
\end{align*}
\end{thm}

\section{The Eichler--Selberg trace formula for $\Gamma(p^rN,M)$}\label{secMF}
The Eichler--Selberg trace formula in the case where $(N, q) = 1$ is given in Theorem \ref{trace_gamma_N,M} (see \cite[Theorem 9]{Kaplan and Petrow:2017}). In this section, we obtain the Eichler--Selberg trace formula for the congruence subgroup $\Gamma(p^rN, M)$ of $\mathrm{SL}_2(\mathbb{Z})$ using the methods of Kaplan and Petrow.
Trace formula no kantanna senkoukenyuu wo matomeru. Selberg, Hijikata, Oesterle, Cohen no syoukai
\subsection{Eichler--Selberg trace formula for $\Gamma(p^rN,M), r \geq 1$}
\begin{thm}[\textbf{Eichler--Selberg trace formula for $\Gamma(p^rN,M)$}]\label{trace_gamma_p^rN,M}
     Let $r,M,N \geq 1$ be positive integers, $p$ be a prime and $q$ be a power of $p$. Let $d \in (\mathbb{Z}/p^rN\mathbb{Z})^{\times}$ and $k \geq 2$ be an integer. Assume that $M \mid N$, $(N,q)=1$, and $d^2 q \equiv 1 \ (\bmod{M})$. Let $T_q$ be the Hecke operator on $S_k(\Gamma(p^rN,M))$, $\langle d \rangle$ be the diamond operator on $S_k(\Gamma(p^rN,M))$.
     Let $L=(d^2 q -1,N)$. Then, we have the following formula.
    \begin{equation*}
\begin{split}
    T_{\mathrm{trace}}(p^rN,M) &= \mathrm{Tr}(T_q \langle d \rangle \mid S_k(\Gamma(p^rN,M)))\\
    &=  T_{\mathrm{id}}(p^rN,M) - T_{\mathrm{ell}}(p^rN,M) - T_{\mathrm{hyp}}(p^rN,M) + T_{\mathrm{dual}}(p^rN,M),
\end{split}
\end{equation*}
where $T_{\mathrm{id}}, T_{\mathrm{ell}}, T_{\mathrm{hyp}}, T_{\mathrm{dual}}$ are respectively defined by  
\begin{align*}
    T_{\mathrm{id}}(p^rN,M)&:= \varphi(p^rN) \cdot\frac{k-1}{24}q^{k/2-1} \psi(p^rNM) (\delta_{p^rN} (q^{1/2}d,1) + (-1)^k \delta_{p^rN}(q^{1/2}d,-1)),\\
    T_{\mathrm{ell}}(p^rN,M) &:= \varphi(p^rN) \cdot \frac{\psi(p^rN^2)}{\psi(p^rN^2/M^2)}\sum_{\Lambda \mid (L/M)} \frac{\varphi(\Lambda^2)\varphi(p^rN/(M\Lambda))}{\varphi(p^rN/M)} \\
    &\quad \sum_{t^2 < 4q} U_{k-2}(t,q) H_{p^rN,\Lambda M}(t,q,d), \\
    T_{\mathrm{hyp}}(p^rN,M) &:= \varphi(p^rN) \cdot \frac{1}{4} \sum_{b \mid q} \min(b,q/b)^{k-1} \sum_{\substack{\tau \mid p^rNM \\ g \mid (b-q/b)}}\bigg( \frac{\varphi(g) \varphi(p^rN(M,g)/g)}{\varphi(N)}\\
            &\quad \times (\delta_{p^rN(M,g)/g} (y_{\tau}d,1) + (-1)^k \delta_{p^rN(M,g)/g}(y_\tau d ,-1)) \bigg),\\
    T_{\mathrm{dual}}(p^rN,M) &:=\sigma(q)\delta(k,2).
\end{align*}
\end{thm}

\subsection{Proof of Theorem \ref{trace_gamma_p^rN,M} (Step 1)}
By Theorem \ref{trace_gamma_0}, for a Dirichlet character $\chi$ modulo $p^rN$, we have
\begin{equation}\label{T_qS_k(Gamma(p^rMN,chi))}
    \mathrm{Tr}(T_q \mid S_k(\Gamma_1 (p^rMN),\chi)) = T_\chi^{(i)} - T_\chi^{(e)} - T_\chi^{(h)} + T_\chi^{(d)}.
\end{equation}
Here, for an integer $t$ with $t^2- 4q < 0$, we define $T_{\chi}^{(e)}(t)$ by
\[ T_\chi^{(e)}(t) := \sum_{m^2 \mid (t^2-4q)} h_w \left(\frac{t^2-4q}{m^2}\right)\mu_\chi(t,m,q),\]
and $T_{\chi}^{(i)},T_{\chi}^{(e)},T_{\chi}^{(h)},T_{\chi}^{(d)}$ by
\begin{align*}
    T_\chi^{(i)} &:= \frac{k-1}{12} \psi(p^rMN)\chi(q^{1/2})q^{k/2 -1},\\
    T_\chi^{(e)} &:= \frac{1}{2}\sum_{t^2<4q}U_{k-2}(t,q)T_{\chi}^{(e)}(t),\\
    T_\chi^{(h)} &:= \frac{1}{2} \sum_{b \mid q}\min (b,q/b)^{k-1}\sum_{\tau}\varphi((\tau,p^rMN/\tau))\chi(y_\tau),\\
    T_\chi^{(d)} &:= \delta (k,2) 1_{\chi=1} \sum_{\substack{c \mid q \\ (p^rMN,q/c) = 1}} c.
\end{align*}

If $\chi(-1) \ne (-1)^k$, we have $S_k(\Gamma(p^rN,M),\chi) = \{0\}$.
Therefore, 
by (\ref{sum_ch}) and (\ref{T_qS_k(Gamma(p^rMN,chi))}), we have
\begin{equation}
    \frac{1}{\varphi(p^rN)} \mathrm{Tr} (\langle d \rangle T_q \mid S_k(\Gamma(p^rN,M))) = T^{(i)} - T^{(e)} - T^{(h)} + T^{(d)}.
\end{equation}
Here, we define $T^{(i)},T^{(e)},T^{(h)},T^{(d)}$ by
\begin{align}
    T^{(i)} &:= \frac{1}{\varphi(p^rN)} \sum_{\substack{\chi \bmod{p^rN} \\ \chi(-1) = (-1)^k}} \chi(d) T_{\chi}^{(i)},\\
    T^{(e)} &:= \frac{1}{\varphi(p^rN)} \sum_{\substack{\chi \bmod{p^rN} \\ \chi(-1) = (-1)^k}} \chi(d) T_{\chi}^{(e)}, \label{T^{(e)}}\\
    T^{(h)} &:= \frac{1}{\varphi(p^rN)} \sum_{\substack{\chi \bmod{p^rN} \\ \chi(-1) = (-1)^k}} \chi(d) T_{\chi}^{(h)},\\
    T^{(d)} &:= \frac{1}{\varphi(p^rN)} \sum_{\substack{\chi \bmod{p^rN} \\ \chi(-1) = (-1)^k}} \chi(d) T_{\chi}^{(d)}.
\end{align}

It is easy to see that $\varphi(p^rN)T^{(i)}, \varphi(p^rN)T^{(h)}, \varphi(p^rN)T^{(d)}$ coincide with the terms $T_{\mathrm{id}}(p^rN,M), T_{\mathrm{hyp}}(p^rN,M), T_{\mathrm{dual}}(p^rN,M)$ in Theorem \ref{trace_gamma_p^rN,M} (for the proof, see \cite[Section 4.1, 4.2, 4.4]{Kaplan and Petrow:2017}, respectively).
In this paper, we only calculate the elliptic term $T^{(e)}$.

\subsection{Proof of Theorem \ref{trace_gamma_p^rN,M}  (Step 2): Calculation of the elliptic term}\label{Ellitpic_Term}
In the following, we put $\Delta := t^2 -4q$. Note that we assume $M \mid N$, $d \in (\mathbb{Z} /p^rN \mathbb{Z})^{\times}$, and $d^2 q \equiv 1 \ (\bmod{M})$. 

We put
\begin{equation*}
    T^{(e)} (t):=\frac{1}{\varphi(p^rN)} \sum_{\substack{\chi \bmod{p^rN} \\ \chi(-1) = (-1)^k}} \chi (d)T_{\chi}^{(e)}(t).
\end{equation*}
Changing the order of summation, we have
\begin{equation}\label{T^{(e)}(t)}
    T^{(e)} (t) = \sum_{m^2 \mid \Delta} h_w \left( \frac{\Delta}{m^2} \right) \frac{1}{\varphi(p^rN)} \sum_{\substack{\chi \bmod{p^rN} \\ \chi(-1) = (-1)^k}} \chi(d) \mu_{\chi} (t,m,q).
\end{equation}

We put
\begin{align}\label{def:W_N_M_n}
    W_{N,M,m}(d) &:= \sum_{c \in S(MN,m,t,q)} \delta_N (c,d^{-1})\\
    C_{N,M}(t,q,d) &:= \sum_{m^2 \mid \Delta} h_w \left( \frac{\Delta}{m^2} \right) \frac{\psi(MN)}{\psi(MN/(MN,m))}W_{N,M,m}(d).
\end{align}
Then, we have
\begin{equation*}
\begin{split}
    &\frac{1}{\varphi(p^rN)} \sum_{\substack{\chi \bmod{p^rN} \\ \chi(-1) = (-1)^k}} \chi(d) \mu_{\chi} (t,m,q) \\
    &= \frac{\psi(p^rMN)}{\psi(p^rMN/(p^rMN,m))} \sum_{c \in S(p^rMN,m,t,q)} \frac{1}{\varphi(p^rN)} \sum_{\substack{\chi \bmod{p^rN} \\ \chi(-1) = (-1)^k}} \chi(dc)\\
    &= \frac{\psi(p^rMN)}{\psi(p^rMN/(p^rMN,m))} \dfrac{1}{2} \left( W_{p^rN,M,m}(d) + (-1)^k W_{p^rN,M,m}(-d) \right).
\end{split}
\end{equation*}
Thus, the equation (\ref{T^{(e)}(t)}) can be rewritten as follows.
\begin{equation*}
    T^{(e)}(t) = \frac{1}{2} \left( C_{p^rN,M}(t,q,d) + (-1)^k C_{p^rN,M}(t,q,-d) \right)
\end{equation*}

Since \( W_{p^r,1,m}(d) \) will be used in the calculation of $C_{p^rN,M}(t,q,d)$, we organize the results.
\begin{lem}\label{W_p^rN,1,m(d)}
    Let $d \in (\mathbb{Z} /p^r \mathbb{Z})^{\times}$ and $m \in \mathbb{Z}$ with $m^2 \mid (t^2 -4q)$. We have
    \begin{equation*}
        W_{p^r,1,m}(d) = D(t;p^r).
    \end{equation*}
\end{lem}
\begin{proof}
    By (\ref{def:W_N_M_n}), we have
    \begin{equation*}
        W_{p^r,1,m}(d) = \sum_{c \in S(p^r,m,t,q)} \delta_{p^r}(c,d^{-1}).
    \end{equation*}
    We consider the following three cases separately.
\begin{itemize}
    \item Assume that $p \mid t$. Let $c \in S(p^r,m,t,q)$. This means that a lift $\tilde{c} \in (\mathbb{Z}/p^r(p^r,m)\mathbb{Z})^\times$ of $c$ satisfies
    \begin{equation*}
        \tilde{c}^2 -t\tilde{c} + q \equiv 0 \ (\bmod{p^r(p^r,m)}).
    \end{equation*}
    Since $t \equiv q \equiv 0 \ (\bmod{p})$, we have $\tilde{c} \equiv c \equiv 0 \ (\bmod{p})$.
    This contradicts $c \in (\mathbb{Z} /p^r \mathbb{Z})^{\times}$. Thus the set $S(p^r,m,t,q)$ is empty and $W_{p^r,1,m}(d) = 0$. On the other hand, for any element $d \in (\mathbb{Z}/p^r\mathbb{Z})^{\times}$, the congruence
    \[dq  + d^{-1} \equiv t \ (\bmod{p^r})\]
    cannot hold
    because $t \equiv q \equiv 0 \ (\bmod{p})$. Hence $D(t;p^r)=0$. Therefore, when $p \mid t$, we have $W_{p^r,1,m}(d)= 0 = D(t;p^r)$. 
    \item Assume that $p \nmid t$ and $p^r \mid q$. Since $m^2 \mid (t^2 - 4q)$, we have $p \nmid m$.
    Thus we have $(p^r,m) = 1$. 
    Let $c \in S(p^r,m,t,q)$. 
    Since $(p^r,m )= 1$, the element $c$ satisfies
    \begin{equation*}
        c^2 -tc + q \equiv 0 \ (\bmod{p^r}).
    \end{equation*}
    Since $p^r \mid q$, we have $c(c-t) \equiv 0 \ (\bmod{p^r})$.
    Since $c \in (\mathbb{Z} /p^r\mathbb{Z})^{\times}$, we have  $c \equiv t \ (\bmod{p^r})$.
    Hence we have $W_{p^r,1,m}(d) = \delta_{p^r}(t,d^{-1})$. Since $p^r \mid q$, we have
    \begin{align*}
        &\delta_{p^r}(t,d^{-1}) = 1 \\
        &\iff t\equiv d^{-1} \ (\bmod{p^r})\\
        &\iff t \equiv dq + d^{-1} \ (\bmod{p^r})\\
        &\iff D(t;p^r) = 1.
    \end{align*}
    We have $\delta_{p^r}(t,d^{-1}) = D(t;p^r)$. Therefore, when $p \nmid t$ and $p^r \mid q$,  we have $W_{p^r,1,m}(d)= \delta_{p^r}(t,d^{-1}) = D(t;p^r)$. 
    \item Assume that $p \nmid t$ and $p^r \nmid q$. Since $m^2 \mid (t^2 - 4q)$, we have $p \nmid m$.
    Thus we have $(p^r,m) = 1$.  Let $c \in S(p^r,m,t,q)$. 
    Since $(p^r,m )= 1$, the element $c$ satisfies
    \begin{equation}\label{c^2-tc+q_pmod(p^r)}
        c^2 -tc + q \equiv 0 \ (\bmod{p^r}).
    \end{equation}
    Since $q$ is a power of $p$ and $p^r \nmid q$, we have $q \mid p^r$. Hence, $c$ satisfies $c(c-t) \equiv 0 \ (\bmod{q})$.
    Since $c \in (\mathbb{Z}/p^r\mathbb{Z})^{\times}$, we have $c \equiv t \ (\bmod{q})$.
    Then, we write $c \equiv t + qi \ (\bmod{p^r})$ for some $0 \leq i < p^r/q$.
    Substituting this expression into (\ref{c^2-tc+q_pmod(p^r)}), we obtain
    \begin{equation*}
        q(qi^2 + ti + 1) \equiv 0 \ (\bmod{p^r}).
    \end{equation*}
    Consider the map $f:\mathbb{Z}/(p^r/q)\mathbb{Z} \to \mathbb{Z}/(p^r/q)\mathbb{Z}$, $i \mapsto qi^2 + ti + 1$. 
    For integers $i,j$ with $0 \leq i,j < p^r/q$, since $p \nmid t$, we have
\begin{align*}   
    &qi^2 + ti + 1 \equiv qj^2 + tj + 1 \ (\bmod{p^r/q})\\
    &\iff (i-j)(q(i+j)+t) \equiv 0 \ (\bmod{p^r/q}) \\
    &\iff i \equiv j \ (\bmod{p^r/q}).
\end{align*}
    Therefore $f$ is an injective map between finite sets of same order. Hence, $f$ is bijective and we have $\#S(p^r,m,t,q) = 1$.
    There exists a unique integer $0 \leq i' < p^r/q$ satisfying $q(qi'^2 + ti' + 1) \equiv 0 \ (\bmod{p^r})$.
    We put $c' := t + qi' \in (\mathbb{Z}/p^r\mathbb{Z})$. Then $c'$ is a unique element of $S(p^r,m,t,q)$.
    Hence, we have $W_{p^r,1,m}(d) = \delta_{p^r}(c',d^{-1})$. Since $d \in (\mathbb{Z}/p^r\mathbb{Z})^\times$, we have 
    \begin{align*}
        &\delta_{p^r}(c' ,d^{-1}) = 1 \\
        &\iff c' \equiv d^{-1} \ (\bmod{p^r})\\
        &\iff d^{-2} - td^{-1} + q \equiv 0 \ (\bmod{p^r})\\
        &\iff t \equiv dq + d^{-1} \ (\bmod{p^r})\\
        &\iff D(t;p^r) = 1.
    \end{align*}
    Therefore, when $p \nmid t$ and $p^r \nmid q$, we have $W_{p^r,1,m}(d)=\delta_{p^r}(c',d^{-1}) = D(t;p^r)$.
\end{itemize}    
\end{proof}

\begin{lem}\label{C_1,p^rN,M(t,q,d)}
Let $M,N$ be positive integers with $M \mid N$ and $(N,q)=1$. Let $d \in (\mathbb{Z} /p^rN\mathbb{Z})^{\times}$ with $d^2q \equiv 1 \ (\bmod{M})$. We have
    \begin{equation}
        C_{p^rN,M}(t,q,d) = D(t;p^r) \cdot C_{N,M}(t,q,d).
    \end{equation}
\end{lem}
\begin{proof}
    Since $(N,q)=1$, each $m$ can be expressed as $m=np^{\kappa}$ where $n \geq 1$ is a positive integer coprime to $p$ and $\kappa = v_p(m)$.
     Using the multiplicativity of $\psi$, we have
    \begin{align*}
        \psi(p^rMN) &= \psi(MN) \cdot \psi(p^r), \\
        (p^rMN,m) &= (MN,n) \cdot (p^r,p^\kappa).
    \end{align*}
    Using the multiplicativity of $\delta_{p^rN}(c,d^{-1})$ and the Chinese remainder theorem, we rewrite $W_{p^rN,M,m}(d)$ as follows.
    \begin{equation*}
    \begin{split}
        W_{p^rN,M,m}(d) &= \sum_{c \in S(p^rNM,m,t,q)} \delta_{p^rN}(c,d^{-1}) \\
        &= \sum_{c \in S(p^rMN,m,t,q)} \delta_{N}(c,d^{-1}) \cdot \delta_{p^r}(c,d^{-1}) \\
        &= \sum_{c_1 \in S(MN,n,t,q)} \ \sum_{c_2 \in S(p^r,p^\kappa,t,q)} \delta_{N}(c_1,d^{-1}) \cdot \delta_{p^r}(c_2,d^{-1}) \\
        &= W_{N,M,n}(d) \cdot W_{p^r,1,p^\kappa}(d).
    \end{split}
    \end{equation*}
Therefore, we have
    \begin{equation*}
        \begin{split}
            C_{p^rN,M}(t,q,d) &= \sum_{m^2 \mid \Delta} h_w \left( \frac{\Delta}{m^2} \right) \frac{\psi(p^rMN)}{\psi(p^rMN/(p^rMN,m))} W_{p^rN,M,m}(d)\\
            &= \sum_{\substack{n^2 \mid \Delta  \\ v_p(n) = 0} } \bigg( \frac{\psi(MN)}{\psi(MN/(MN,n))} 
            W_{N,M,n}(d) \\
            &\qquad \quad \cdot \sum_{\kappa \geq 0}
            h_w \left( \frac{\Delta}{n^2 p^{2\kappa}} \right)
        \frac{\psi(p^r)}
        {\psi(p^{r- \min(\kappa,r)})}
        W_{p^r,1,p^\kappa}(d) \bigg). 
        \end{split}
    \end{equation*}
    
    When $p \mid t$, by Lemma \ref{W_p^rN,1,m(d)}, since $W_{p^r,1,p^{\kappa}}(d)=0$, we have $C_{p^rN,M}(t,q,d) = 0$.
    
    When $p \nmid t$, since $m \mid t^2-4q$, there does not exist $m$ such that $v_p(m) = \kappa > 0$. Thus, by Lemma \ref{W_p^rN,1,m(d)}, we have $W_{p^r,1,1}(d)=D(t;p^r)$ and we have
    \begin{equation*}
         C_{p^rN,M}(t,q,d) = D(t;p^r) \sum_{\substack{n^2 \mid \Delta  \\ v_p(n) = 0} } h_w \left( \frac{\Delta}{n^2} \right) \frac{\psi(MN)}{\psi(MN/(MN,n))} 
            W_{N,M,n}(d).
    \end{equation*}
    Furthermore, since $p \nmid t$, we have $p \nmid \Delta$ and we obtain the desired result.
    \begin{equation*}
    \begin{split}
        C_{p^rN,M}(t,q,d) &= D(t;p^r) \sum_{n^2 \mid \Delta} h_w \left( \frac{\Delta}{n^2} \right) \frac{\psi(MN)}{\psi(MN/(MN,n))} 
            W_{N,M,n}(d) \\
            &= D(t;p^r) C_{N,M}(t,q,d).
    \end{split}
    \end{equation*}
\end{proof}

The desired elliptic term $T^{(e)}$ is given by (\ref{T^{(e)}}) as
\begin{equation*}
    T^{(e)} = \frac{1}{2} \sum_{t^2<4q} U_{k-2}(t,q) T^{(e)}(t).
\end{equation*}
Using  Lemma \ref{W_p^rN,1,m(d)} and Lemma \ref{C_1,p^rN,M(t,q,d)}, we obtain the following result.
\begin{prop}\label{T^{(e)}(p^rN,M)}
    Let $r,M,N \geq 1$ be positive integers with $M \mid N$ and $(q,N) = 1$. Let $d \in (\mathbb{Z} /p^rN\mathbb{Z})^{\times}$ with $d^2q \equiv 1 \ (\bmod{M})$. Let $L := (d^2q-1,p^rN)$. We have
    \begin{equation*}
        T^{(e)} = \frac{\psi((p^rN)^2)}{\psi((p^rN)^2/M^2)} \sum_{\Lambda \mid (L/M)} \frac{\varphi(\Lambda^2)\varphi(p^rN/\Lambda M)}{\varphi(p^rN/M)} \sum_{t^2<4q} U_{k-2}(t,q)H_{p^rN,\Lambda M}(t,q,d).
    \end{equation*}
\end{prop}
\begin{proof}
    Since $(d^2q-1,p^r)=1$, we have $L=(d^2q-1,N)$. By \cite[Equation (98)]{Kaplan and Petrow:2017}, we have
        \begin{equation*}
        C_{N,M}(t,q,d) = 2 \frac{\psi(N^2)}{\psi(N^2/M^2)} \sum_{\Lambda \mid (L/M)} \frac{\varphi(\Lambda^2) \varphi(N/(M \Lambda))}{\varphi(N/M)} H_{N,\Lambda M}(t,q,d).
    \end{equation*}
    By Lemma \ref{C_1,p^rN,M(t,q,d)} and the multiplicativity of $\varphi$ and $\psi$, we have
    \begin{equation*}
        C_{p^rN,M}(t,q,d) = 2D(t;p^r) \frac{\psi((p^rN)^2)}{\psi((p^rN)^2/M^2)} \sum_{\Lambda \mid (L/M)} \frac{\varphi(\Lambda^2) \varphi(p^rN/(M \Lambda))}{\varphi(p^rN/M)} H_{N,\Lambda M}(t,q,d).
    \end{equation*}
    Thus, by Lemma \ref{H_p^rn_1,n_2(t,q,1)}, we rewrite the elliptic term $T^{(e)}$ as follows.
    \begin{equation}\label{T^(e)_p^rN,M}
    \begin{split}
        T^{(e)} &= \frac{\psi((p^rN)^2)}{\psi((p^rN)^2/M^2)} \sum_{\Lambda \mid (L/M)} \frac{\varphi(\Lambda^2)\varphi(p^rN/\Lambda M)}{\varphi(p^rN/M)} \\
        &\quad \cdot \frac{1}{2}\sum_{t^2<4q} U_{k-2}(t,q)\left( H_{p^rN,\Lambda M}(t,q,d) +(-1)^k H_{p^rN,\Lambda M}(t,q,-d) \right).
    \end{split}
    \end{equation}
    The parity of the function $U_{k-2}(t,q)$ as a function of $t$ matches the parity of $k$. Moreover, by Definition \ref{H_n_1,n_2}, we have $H_{p^rN,\Lambda M}(t,q,-d) = H_{p^rN,\Lambda M}(-t,q,d)$. Therefore the following equation holds:
\begin{equation*}
    U_{k-2}(t,q) H_{p^rN,\Lambda M}(t,q,-d) = (-1)^{k-2} U_{k-2}(-t,q) H_{p^rN,\Lambda M}(-t,q,d).
\end{equation*}
The following transformation can be made:
\begin{equation*}
    \begin{split}
        &\sum_{t^2<4q} U_{k-2}(t,q)\left(H_{p^rN,\Lambda M}(t,q,d) +(-1)^k H_{p^rN,\Lambda M}(t,q,-d) \right) \\
        &= \sum_{t^2 < 4q} \left( U_{k-2}(t,q) H_{p^rN,\Lambda M}(t,q,d) + U_{k-2}(-t,q)H_{p^rN,\Lambda M}(-t,q,d) \right) \\
        &= 2 \sum_{t^2 < 4q} U_{k-2}(t,q)H_{p^rN,\Lambda M}(t,q,d).
    \end{split}
\end{equation*}
\end{proof}
Thus, the proof of Theorem \ref{trace_gamma_p^rN,M} is complete.

\section{Elliptic curves over a finite field with a specified subgroup}\label{secEC}

In this section, we first summarize the results of Kaplan and Petrow \cite[Section 3]{Kaplan and Petrow:2017}, and then calculate the moment $\mathbb{E}_q(U_{k-2}(t_E,q)\Phi_A)$ in the case where $p \mid \#A$.

    Let $A$ be a finite abelian group, and $t \in \mathbb{Z}$.
    Recall that $\mathcal{C}$ denotes the set of $\mathbb{F}_q$--isomorphism classes of elliptic curves over $\mathbb{F}_q$
    We define the subset of $\mathcal{C}$ as follows.
    \begin{equation*}
            \mathcal{C}(A,t) :=\{E/\mathbb{F}_q :\text{there exists an injective homomorphism } A \hookrightarrow E(\mathbb{F}_q) \ \text{and }t_E = t\}.
    \end{equation*}
By Hasse's Theorem, we have $t_E^2 \leq 4q$. We rewrite $\mathbb{E}_q(U_{k-2}(t_E,q)\Phi_A))$ as follows.
\begin{equation}\label{E_q_A}
    \begin{split}
        \mathbb{E}_q(U_{k-2}(t_E,q)\Phi_A) &= \frac{1}{q}\sum_{\substack{E/\mathbb{F}_q \\ A \hookrightarrow E(\mathbb{F}_q)}} \frac{U_{k-2}(t_E,q)}{\# \mathrm{Aut}_{\mathbb{F}_q}(E)} \\
        &= \sum_{t^2 \leq 4q} U_{k-2}(t_E,q) \left( \frac{1}{q} \sum_{\substack{E/\mathbb{F}_q\\ t_E = t \\ A \hookrightarrow E(\mathbb{F}_q)}} \frac{1}{\# \mathrm{Aut}_{\mathbb{F}_q}(E)} \right) \\
        &= \sum_{t^2 \leq 4q}U_{k-2}(t,q) \mathbb{P}_q (\mathcal{C}(A,t)).
    \end{split}
\end{equation}
Here we define $\mathbb{P}_q (\mathcal{C}(A,t))$ by 
\begin{equation*}
    \mathbb{P}_q (\mathcal{C}(A,t)) :=
 \frac{1}{q} \sum_{\substack{E/\mathbb{F}_q\\ t_E = t \\ A \hookrightarrow E(\mathbb{F}_q)}} \frac{1}{\# \mathrm{Aut}_{\mathbb{F}_q}(E)}. \\
\end{equation*}

Deuring \cite{Deuring:1941}, Waterhouse \cite{Waterhouse:1969}, Lenstra \cite{Lenstra:1977}, and Schoof \cite{Schoof:1987} obtained results expressing the number of isomorphism classes of elliptic curves in terms of class numbers, and Kaplan and Petrow used Schoof's result to obtain the following conclusion.

\begin{thm}[Kaplan--Petrow {\cite[Theorem 7]{Kaplan and Petrow:2017}}]\label{P_q_ordinary}
    Let $A$ be a finite abelian group such that $A \cong \mathbb{Z} / n_1\mathbb{Z} \times \mathbb{Z} / n_2\mathbb{Z}$ where $n_1,n_2 \geq 1$ are positive integers satisfying $n_2 \mid n_1$. Then, we have
    \begin{equation*}
        \mathbb{P}_q(\mathcal{C}(A,t)) =
        \begin{cases}
            \displaystyle \frac{1}{q}H_{n_1,n_2}(t,q,1) & (\text{if }p \nmid t \ \text{and } t^2 < 4q)\\*[3mm]
            0 & (\text{otherwise}). 
        \end{cases}
    \end{equation*}
\end{thm}

As an application of Theorem \ref{P_q_ordinary}, we prove an explicit formula for $\mathbb{E}_q(\mathcal{C}(A,t))$ when the order $\#A$ is divisible by $p$.
In this section, we take $d = 1$ in the definition of $D(t;n)$.
Let $A$ be a finite abelian group such that $A \cong \mathbb{Z} / p^rn_1\mathbb{Z} \times \mathbb{Z} / n_2\mathbb{Z}$ where $n_1,n_2 \geq 1$ are positive integers satisfying $n_2 \mid n_1$ and $r \geq 1$. 
We consider two cases separately.
\begin{itemize}
    \item Assume that $p \nmid t$. 
By Theorem \ref{P_q_ordinary}, we have
\begin{equation*}
    \mathbb{P}_q(\mathcal{C}(A,t)) = \frac{1}{q} H_{p^rn_1,n_2}(t,q,1).
\end{equation*}
    \item Assume that $p \mid t$. If  an elliptic curve $E / \mathbb{F}_q$ has an injective homomorphism $A \hookrightarrow E(\mathbb{F}_q)$, we have $p \mid \#E(\mathbb{F}_q)$. This induces $p$ does not divide $t_E = q+1 - \#E(\mathbb{F}_q)$. Thus we have
\begin{equation*}
    \mathbb{P}_q(\mathcal{C}(A,t)) = 0.
\end{equation*}
\end{itemize}

From the above results, we have the following lemma.
\begin{lem}\label{E_q(p^rn_1,n_2)}
Let $A \cong \mathbb{Z}/p^rn_1\mathbb{Z} \times \mathbb{Z} / n_2 \mathbb{Z}$ be a finite abelian group. Assume that $r \geq 1$ and $n_1,n_2$ are positive integers satisfying $n_2 \mid n_1$ and $(q,n_1) = 1$. We have
    \begin{equation*}
    \mathbb{E}_q(U_{k-2}(t_E,q)\Phi_{A}) = \frac{1}{q} \sum_{t^2<4q} U_{k-2}(t,q)H_{p^rn_1,n_2}(t,q,1).
    \end{equation*}
\end{lem}

\section{Proof of Theorem \ref{MainTheorem}}
In this section, we prove Theorem \ref{MainTheorem} by applying the Eichler--Selberg trace formulas for $\Gamma(p^rn_1,\lambda)$ for several $\lambda$.

In Section \ref{secMF}, we showed that a part of the trace $\mathrm{Tr}(T_q \langle d \rangle \mid S_k (n_1,n_2))$ can be expressed as a sum of class numbers. Focusing on the part corresponding to the sum of class numbers and organizing it, by Proposition \ref{T^{(e)}(p^rN,M)}, we have
\begin{equation*}
    T_{p^r n_1, \lambda}(q,d) = \sum_{\Lambda \mid (L / \lambda)} \varphi(\Lambda^2) \varphi\left( \frac{p^r n_1}{\lambda \Lambda} \right) \sum_{t^2 < 4q} U_{k-2}(t,q,d) H_{p^r n_1, \lambda \Lambda}(t,q,d).
\end{equation*}

Since the function $\phi(n)$ is the inverse function of the function $\varphi(n^2)$ with respect to the Dirichlet convolution, the following relation holds.
\begin{lem}\label{sumT_p^rn_1,n_2}
    Let $r,n_1,n_2 \geq 1$ be positive integers with $n_2 \mid n_1$ and $(q,n_1)=1$. Let $d \in (\mathbb{Z}/p^rn_1\mathbb{Z})^\times$ with $d^2 q \equiv 1 \ (\bmod{n_2})$. Let $L = (d^2 q - 1, p^r n_1)$. We have
    \begin{equation}\label{sum_Tr}
    \frac{1}{\varphi(p^r n_1 / n_2)} \sum_{\nu \mid (L / n_2)} \phi(\nu) T_{p^r n_1, n_2 \nu}(q,d) = \sum_{t^2 < 4q} U_{k-2}(t,q) H_{p^r n_1, n_2}(t,q,d).
    \end{equation}
\end{lem}
\begin{proof}
By Proposition \ref{T^{(e)}(p^rN,M)} the left hand side is equal to
\begin{equation*}
    \frac{1}{\varphi(p^rn_1/n_2)} \sum_{\nu \mid (L/n_2)} \phi(\nu) \sum_{\Lambda \mid (L/n_2 \nu)} \varphi(\Lambda^2) \varphi \left( \frac{p^rn_1}{n_2 \nu \Lambda} \right) \sum_{t^2<4q} U_{k-2}(t,q) H_{p^rn_1,n_2 \nu \Lambda}(t,q,d).
\end{equation*}
Since $\nu \Lambda \mid (L/n_2)$, by setting $\mu := \nu \Lambda$ and changing the order of summation, we obtain the following:
\begin{equation*}
    \frac{1}{\varphi(p^rn_1/n_2)} \sum_{\mu \mid (L/n_2)} \varphi \left( \frac{p^rn_1}{n_2 \mu} \right) \sum_{t^2<4q} U_{k-2}(t,q,d) H_{p^rn_1,n_2 \mu}(t,q,d) \sum_{\nu \Lambda = \mu} \phi(\nu) \varphi(\Lambda^2).
\end{equation*}

Since the function $\phi(n)$ is the inverse function of the function $\varphi(n^2)$ with respect to Dirichlet convolution, 
we have
\[\sum_{\nu \Lambda = \mu} \phi(\nu) \varphi(\Lambda^2) =
\begin{cases}
     1 & (\text{if } \mu = 1)\\
     0 & (\text{otherwise}).
\end{cases}\]
Therefore, only the part where $\mu = 1$ remains, and we obtain the following result:
\begin{equation*}
\begin{split}
    \frac{1}{\varphi(p^rn_1/n_2)} \varphi \left( \frac{p^rn_1}{n_2} \right) \sum_{t^2 < 4q}U_{k-2}(t,q) H_{p^rn_1,n_2}(t,q,d) \\= \sum_{t^2 < 4q}U_{k-2}(t,q)H_{p^rn_1,n_2}(t,q,d).
\end{split}
\end{equation*}
\end{proof}

By Lemma \ref{E_q(p^rn_1,n_2)} and Lemma \ref{sumT_p^rn_1,n_2}, the proof of Theorem \ref{MainTheorem} is complete.

\section*{Acknowledgements}
The author would like to thank Nathan Kaplan and Ian Petrow for invaluable comments and suggestions on an earlier draft of this paper.
The author would like to thank his advisor, Tetsushi  Ito, for useful discussions and warm encouragements.

\section*{Declarations}

\noindent
\textbf{Author contributions} T.K.\ is the sole author and completed all work related to this manuscript.

\vspace*{2mm}

\noindent
\textbf{Data availability} No datasets were generated or analysed during the current 
study.

\vspace*{2mm}

\noindent
\textbf{Conflict of interest}
The authors declare no conflict of interest.

\end{document}